\documentclass[a4paper,10pt]{article}
\usepackage{amssymb}
\usepackage{amsmath,amsthm}
 \newtheorem{lemma}{Lemma}
 \newtheorem{theorem}{Theorem}

\begin{document}
\title{Connectivity of direct products of graphs}

\author{Wei Wang \footnote{Corresponding author: wangwei.math@gmail.com}\hspace{1cm}Ni-Ni Xue\\
\small College of Information Engineering, Tarim University, \\
\small Alar, Xinjiang, 843300, P.R.China}
\date{}
 \maketitle

\begin{abstract}
Let  $\kappa(G)$ be the connectivity of $G$ and $G\times H$ the
direct product of $G$ and $H$. We prove that for any graphs $G$ and
$K_n$ with $n\ge 3$, $\kappa(G\times
K_n)=\mbox{min}\{n\kappa(G),(n-1)\delta(G)\}$, which was conjectured
by Guji and Vumar.

\end{abstract}


\section{Introduction}

  Throughout this paper we consider only finite undirected graphs
without loops and multiple edges.

  Let $G=(V(G),E(G))$ be a graph. The connectivity of $G$ is the
number, denoted as $\kappa(G)$, equal to the fewest number of
vertices whose removal from $G$ results in a disconnected or trivial
graph. The direct (or Kronecker) product $G\times H$ of graph $G$
and $H$ has vertex set $V(G\times H)=V(G)\times V(H)$ and edge set
$E(G\times H)=\{(u_1,v_1)(u_2,v_2):u_1u_2\in E(G)\ \mbox{and}\
v_1v_2\in E(H)\}$.

  The connectivity of direct products of graphs has been studied
  recently. Unlike the case of Cartesian products where the general
  formula was obtained \cite{spac1,xujm}, results for direct products
  have been given only in special cases.  Mamut and Vumar \cite{mamut} considered product of two complete graphs and proved
  for any $K_m$ and $K_n$ with $n\ge m\ge2$ and $n\ge3$,
  \begin{equation}\label{comcom}\kappa(K_m\times K_n)=(m-1)(n-1).\end{equation}
  Later, Guji and Vumar \cite{guji} proved for any bipartite
  graph $G$ and $K_n$ with $n\ge3$,
  \begin{equation}
  \label{bipcom}\kappa(G\times
  K_n)=\mbox{min}\{n\kappa(G),(n-1)\delta(G)\},
  \end{equation}
  where $\delta(G)$ denoted the minimum degree of $G$. In the same
  paper, Guji and Vumar conjectured (\ref{bipcom}) holds even without the
  assumption of bipartiteness of $G$.

In the next section we shall prove the conjecture.
\section{The result}

\begin{theorem}\label{anycom}
$\kappa(G\times
 K_n)=\textup{min}\{n\kappa(G),(n-1)\delta(G)\}$ for $n\ge3$.
\end{theorem}
The proof of the theorem will be postponed to the end of this
section. We first give some properties on direct products of graphs
\cite{bott}.

\begin{lemma}\label{basic}

\textup{(1)} The direct product of nontrivial graphs $G$ and $H$ is
connected if and only if both factors are connected and at least one
factor contains an odd cycle.

\textup{(2)} $\delta(G\times H)=\delta(G)\delta(H)$, and in
particular, $\delta(G\times K_n)=(n-1)\delta(G)$.

\end{lemma}

We shall always label
$V(G)=\{u_1,\ldots,u_m\},V(K_n)=\{v_1,\ldots,v_n\}$ and set
$S_i=\{u_i\}\times V(K_n)$.   Let $S\subseteq V(G\times K_n)$
satisfy the following two   conditions:\\
(1). $|S|<\mbox{min}\{n\kappa(G),(n-1)\delta(G)\}$, and\\
(2). $S'_i:=S_i-S\ne\emptyset$, for $i=1,2,\ldots,m$.

  Associated with  $G,K_n$ and $S$, we define a new graph $G^*$ as follows:\\
  (1). $V(G^*)=\{S'_1,S'_2,\ldots,S'_m\}$, and\\
  (2). $E(G^*)=\{S'_iS'_j: E(S'_i,S'_j)\ne\emptyset\}$,
  where $E(S'_i,S'_j)$ denotes the collec\-tion of all edges in $(G\times K_n-S)$ with one end in $S'_i$ and the
  other in $S'_j$.

  Notice $G^*$ can  be defined only if $\kappa(G)>0$ since otherwise condition
  (1) is meaningless.

\begin{lemma}\label{macro}
If $G$ is connected then $G^*$ is connected.
\end{lemma}

\begin{proof}

Suppose $G^*$ is not connected. Then the vertices of $G^*$ can be
partitioned into two parts, $X^*$ and $Y^*$, such that there are no
edges joining a vertex in $X^*$ and a vertex in $Y^*$.  Let
 $r=|X^*|$. Without loss of generality, we may assume $
X^*=\{S'_1,\ldots,S'_r\}$ and $Y^*=\{S'_{r+1},\ldots,S'_m\}$.

Let $X=\{u_1,\ldots,u_r\}$ and $Y=\{u_{r+1},\ldots,u_m\}$. Since $G$
 is connected, there is at least one edge joining a vertex in $X$ and
a vertex in $Y$. Let $Z$ be the collection of ends of all edges in
$E(X,Y)$.

Let $Z^*=\{S'_j:j\in\{1,\ldots,m\}\  \mbox{and}\ |S'_j|=1\}$. For
each $u_i\in Z$, by the definition of $Z$, there is an edge
$u_iu_j\in E(X,Y)$. It follows that both $S'_i$ and $S'_j$ contains
exactly one element since otherwise  $E(S'_i,S'_j)$ contains at
least one edge by the definition of $G\times K_n$. Therefore
$S'_i\in Z^*$ and we have $|Z|\le|Z^*|$. We need to consider two
cases:

\textbf{Case 1:} Either $X\subseteq Z$ or $Y\subseteq Z$. We may
assume $X\subseteq Z$, then the degree of any vertex $u_i\in X$ can
not exceed $|Z|-1$. Therefore  $\delta(G)\le|Z|-1$. By a simple
calculation, we have\\
\begin{displaymath}
|S|\ge(n-1)|Z^*|\ge(n-1)|Z|>(n-1)\delta(G)\ge\mbox{min}\{n\kappa(G),(n-1)\delta(G)\},
\end{displaymath}\\
a contradiction.

\textbf{Case 2:} $X\not\subseteq Z$ and $Y\not\subseteq Z$. Either
of $X\cap Z$ and $Y\cap Z$ is a separating set. Therefore,
$\kappa(G)\le\mbox{min}\{|X\cap Z|,|Y\cap
Z|\}\le|Z|/2$. Similarly, we have\\
\begin{displaymath}
|S|\ge(n-1)|Z^*|\ge(n-1)|Z|>|Z|n/2\ge
n\kappa(G)\ge\mbox{min}\{n\kappa(G),(n-1)\delta(G)\},
\end{displaymath}\\
again a contradiction.
\end{proof}

While the above lemma tells us that the new graph $G^*$ is
connected, what we most concern is the connectedness of $G\times
K_n-S$. We need the following lemma.

\begin{lemma}\label{micro}
Any vertex of $G^*$, $S'_i$, as a subset of $V(G\times K_n-S)$, is
contained in the vertex set of some component of $G\times K_n-S$.
\end{lemma}

\begin{proof}
It suffices to prove the lemma for $i=1$.

If $|S'_1|=1$, then the assertion holds trivially. We need to
consider two cases:

\textbf{Case 1:} $|S'_1|\ge3$.  Assume $S'_1$ is not contained in
any component of $G\times K_n-S$. Then there must exist a component
$C$ such that $0<|S'_1\cap V(C)|\le |S'_1|/2<|S'_1|-1$. Let
$(u_1,v_s)\in S'_1\cap V(C)$ be any vertex. Since
$|S|<(n-1)\delta(G)$, it follows by (2) of lemma \ref{basic} that
$(u_1,v_s)$ has at least one adjacent vertex in $G\times K_n-S$. Let
$(u_j,v_p)$ be an adjacent vertex of $(u_1,v_s)$. Clearly,
$(u_j,v_p)\in V(C)$ and $S'_1-\{(u_1,v_p)\}\subseteq V(C)$ since
every vertex in $S'_1-\{(u_1,v_p)\}$ is adjacent to $(u_j,v_p)$. It
follows $|S'_1\cap V(C)|\ge |S'_1|-1$,  a contradiction.

\textbf{Case 2:} $|S'_1|=2$. Let $Z^*=\{S'_j:j\in\{1,\ldots,m\} \
\mbox{and}\ |S'_j|=1\}$  and $C^*$ be a component of $G^*-Z^*$
containing $S'_1$. Let $r$ be the order of $C^*$. Without loss of
generality, we may assume $V(C^*)=\{S'_1,\ldots,S'_r\}$.

Since each $S'_j\in V(C^*)$ contains at least two elements, any edge
 $S'_kS'_j$ in $ C^*$ implies every vertex in $S'_k$ has at least one
adjacent vertex in $S'_j$. Therefore, if there is a vertex $S'_j$ in
$C^*$ contained in the vertex set of some component $C$ of $G\times
K_n-S$, then every $S'_k$ is contained in $V(C)$ provided
$S'_kS'_j\in E(C^*)$. It follows by the connectedness of $C^*$ that
$\cup_{i=1}^{r}S'_i\subseteq V(C)$ and hence $S'_1\subseteq V(C)$.

By case1, we may assume each $S'_j\in V(C^*)$ contains exactly two
elements. Let $S'_j=\{u_j\}\times F_j,j=1,\ldots,r$.

 \textbf{Subcase 2.1:}  There exists an edge $S'_jS'_k$ in $C^*$ with
 $F_j\ne F_k$. One easily verify that $S'_j\cup S'_k$ induces a
 connected subgraph of $G\times K_n-S$. The lemma follows.

 \textbf{Subcase 2.2:}  There exists no edge $S'_jS'_k$ in $C^*$ with
 $F_j\ne F_k$. By the connectedness of $C^*$, all $F_j$ in $C^*$ are equal.
 Notice that $C^*$ and the subgraph induced by $\cup_{i=1}^{r}S'_i$ are
 isomorphic to $G[u_1,\ldots,u_r]$ and $G[u_1,\ldots,u_r]\times
 K_2$, respectively. We claim  $G[u_1,\ldots,u_r]$  must contain an odd
 cycle, which will finish our proof by (1) of lemma \ref{basic}.

 Suppose $G[u_1,\ldots,u_r]$ does not contain an odd cycle. Then
 either $r=1$ or $G[u_1,\ldots,u_r]$ is bipartite. Either of the two
 cases implies $\delta(G[u_1,\ldots,u_r])\le r/2$. Let
 $j\in\{1,\ldots,r\}$ such that
 $\mbox{deg}_{G[u_1,\ldots,u_r]}(u_j)=\delta(G[u_1,\ldots,u_r])$.

 Let $u_k$ be any adjacent vertex of $u_j$ in $G$, then either $S'_k\in Z^*$, or $S'_k$ is an adjacent vertex of $S'_j$ in
 $C^*$.  Therefore,
  \begin{equation}\label{del}
 \delta(G)\le \mbox{deg}_G(u_j)=\mbox{deg}_{C^*}(S'_j)+|Z^*|=\mbox{deg}_{G[u_1,\ldots,u_r]}(u_j)+|Z^*|\le r/2+|Z^*|.
 \end{equation}
 By a simple calculation,
 \begin{equation}\label{s}
 |S|\ge (n-2)r+(n-1)|Z^*|\ge(n-1)(\frac{r}{2}+|Z^*|).
 \end{equation}
 From (\ref{del}) and (\ref{s}), we obtain
 \begin{equation}\label{con}
 |S|\ge(n-1)\delta(G),
 \end{equation}
 a contradiction.

\end{proof}

\begin{lemma}\label{red}
Let $m=|G|\ge2$ and $u_i$ be any vertex of $G$. Then\\
\textup{(1)}. $\delta(G-u_i)\ge\delta(G)-1$, and \\
\textup{(2)}. $\kappa(G-u_i)\ge\kappa(G)-1$.
\end{lemma}

\textbf{Proof of Theorem \ref{anycom}}. We apply induction on
$m=|G|$. It trivially holds when $m=1$. We therefore assume $m\ge2$
and that the result holds for all graphs of order $m-1$.

It is clear $\kappa(G\times K_n)\le
\mbox{min}\{n\kappa(G),(n-1)\delta(G)\}$ by lemma \ref{basic}. The
nontrivial part of the proof is hence to show the other inequality.
We may assume $\kappa(G)>0$. Let $S\subseteq V(G\times K_n)$ satisfy
condition (1), i.e., $|S|<\mbox{min}\{n\kappa(G),(n-1)\delta(G)\}$.

\textbf{Case 1:} $S$ satisfies condition (2). It follows by lemma
\ref{macro} and lemma \ref{micro} that $(G\times K_n-S)$ is
connected.

\textbf{Case 2:} $S$ does not satisfy condition (2). Then there
exists an $S_i$ contained in $S$. Therefore, $S-S_i\subseteq
V((G-u_i)\times K_n)$ and

\begin{eqnarray*}
 |S-S_i|&=&|S|-n\\
       &<&\mbox{min}\{n\kappa(G),(n-1)\delta(G)\}-n\\
       &\le&\mbox{min}\{n(\kappa(G)-1),(n-1)(\delta(G)-1)\}\\
       &\le&\mbox{min}\{n\kappa(G-u_i),(n-1)\delta(G-u_i)\}.
\end{eqnarray*}
the last inequality above follows from lemma \ref{red}.

By the induction assumption,
\begin{displaymath}
\kappa((G-u_i)\times
K_n)=\mbox{min}\{n\kappa(G-u_i),(n-1)\delta(G-u_i)\}.
\end{displaymath}
Hence, $(G-u_i)\times K_n-(S-S_i)$ is connected. It follows
by isomorphism that $G\times K_n-S$ is connected.

Either of the two cases implies $(G\times K_n-S)$ is connected.
Thus,
\begin{displaymath}
\kappa(G\times K_n)\ge \mbox{min}\{n\kappa(G),(n-1)\delta(G)\}.
\end{displaymath}
The proof of the theorem is completed by induction.
\subsection*{Acknowledgement}
The authors are much indebted to the referee for his/her valuable
suggestions and corrections that improved the initial version of
this paper.

\end{document}